\documentclass[12pt, reqno]{amsart}
\usepackage{latexsym,amsmath,amsopn,amssymb,amsthm,amsfonts}
\usepackage[T2A]{fontenc}
\usepackage[cp1251]{inputenc}
\usepackage[english]{babel}
\usepackage[backref=page, breaklinks=true,colorlinks=true,linkcolor=blue,citecolor=blue,urlcolor=blue]{hyperref}

\usepackage{graphicx}

\renewenvironment{proof}[1][Proof]{\textbf{#1.} }{\ \rule{0.5em}{0.5em}}

\DeclareMathOperator{\GT}{GT}

\DeclareMathOperator{\sgn}{sgn}

\renewenvironment{proof}[1][Proof]{\textbf{#1.} }
{\ \rule{0.5em}{0.5em}}
\newtheorem{theorem}{Theorem}
\newtheorem{prop}{Proposition}
\newtheorem{lemma}{Lemma}
\newtheorem{corollary}{Corollary}

\theoremstyle{definition}
\newtheorem{definition}{Definition}
\newtheorem{remark}{Remark}

\newtheorem{problem}{Problem}

\begin{document}
\title
[A comprehensive analysis of the Snellius--Pothenot problem]
{A comprehensive analysis of \\ the Snellius--Pothenot problem}
\author{E.V.~Nikitenko, Yu.G.~Nikonorov, M.Q.~Rieck}

\address{Evgeni\u\i\  Vitalievich Nikitenko \newline
Rubtsovsk Industrial Institute of Altai State \newline
Technical University after I.I.~Polzunov \newline
Rubtsovsk, Traktornaya st., 2/6, \newline
658207, Russia}
\email{evnikit@mail.ru}

\address{Yuri\u{\i} Gennadievich Nikonorov \newline
Southern Mathematical Institute of the Vladikavkaz \newline
Scientific Center of the Russian Academy of Sciences, \newline
Vladikavkaz, Vatutina st., 53, \newline
362025, Russia}
\email{nikonorov2006@mail.ru}

\address{Michael Quentin Rieck \newline
Mathematics and Computer Science Department,\newline
Drake University, Des Moines, IA 50311, USA}
\email{michael.rieck@drake.edu}

\begin{abstract}

It is known that a point in three-dimensional Euclidean space whose coordinates are equal to the cosines of the angles $\angle BDC, \angle ADC, \angle ADB$, where the point $D$ lies in the plane of a given triangle $ABC$, lies on the surface $\mathbb{BP}\subset [-1,1]^3$, given by the equation $1+2x_1x_2x_3-x_1^2-x_2^2-x_3^2 = 0$. It should be emphasized that the set of corresponding points essentially depends on the shape of triangle $ABC$. In this paper, we solve the following problem:
For a fixed triangle $ABC$, for each point $U \in \mathbb{BP}$, determine the number of points $D$ from the plane of the triangle with the condition $U=(\cos \angle BDC, \cos \angle ADC, \cos \angle ADB )$. The problem of determining such points $D$ is known as the Snellius--Pothenot problem.

\vspace{2mm}

\vspace{2mm} \noindent Key words and phrases: angle, base, Euclidean plane, Euclidean space, extremal problems, Snellius--Pothenot problem, tetrahedron, triangle.
\end{abstract}

\maketitle

\section{Introduction}

In what follows, $\mathbb{E}^3$ denotes $3$-dimensional Euclidean space and $d(L,M)$ denotes the Euclidean distance between the points $L,M \in \mathbb{E}^3$.

Let us fix a non-degenerate triangle $ABC$ in $\mathbb{E}^3$ with $d(A,B)=\overline{z}$, $d(B,C)=\overline{x}$,  and $d(A,C)=\overline{y}$.
For a tetrahedron $ABCD$ in $\mathbb{E}^3$ with the base $ABC$, we will use the following notation:
$d(D,A)=x$, $d(D,B)=y$, $d(D,C)=z$, $\angle ADB=\overline{\gamma}$, $\angle ADC=\overline{\beta}$, and $\angle BDC=\overline{\alpha}$.
It is easy to check that the numbers $x=d(D,A)$, $y=d(D,B)$, and $z=d(D,C)$  characterize the point $D$,
up to reflection with respect to the plane $ABC$.
\smallskip

Let us consider the set $\Omega (\triangle ABC)$ of all tetrahedra $ABCD$ with a given non-degenerate base $ABC$ in $\mathbb{E}^3$ and $D$ lying outside the plane $ABC$.
Obviously, $\Omega (\triangle ABC)$ is naturally parameterized by points $D \in  \mathbb{E}^3 \setminus \Pi^0$, where $\Pi^0$ is the plane $ABC$.
Moreover, if $\Pi^{-}$ and $\Pi^{+}$ are two distinct half-spaces determined by the plane $\Pi^0$,
then it is sufficient to consider points $D$ in only one of these half-spaces.
\smallskip

\begin{definition}\label{def_numgeom}
For a given triples $(a_1,a_2,a_3)\in (-1,1)^3$, we denote by $N(a_1,a_2,a_3)$  the number of  (geometrically distinct) points $D\in \Pi^{+}$ such that
\begin{equation}\label{eq.numb.1}
\bigl(\cos \angle BDC,\cos \angle ADC,\cos \angle ADB \bigr)=(a_1,a_2,a_3).
\end{equation}
\end{definition}

The following problem is very important:

\begin{problem} \label{prob.1}
For a given triples $(a_1,a_2,a_3)\in (-1,1)^3$, determine the number $N(a_1,a_2,a_3)$.
\end{problem}

Solving this problem requires the use of different approaches, and we plan to devote several future publications to it.
\smallskip

Let us recall the corresponding problem on the plane: Determine the exact location of an unknown observation point on a plane using
only the known coordinates of three given points and the measured angles from the observation point to any pair of given points.
This problem (the so-called {\it Snellius--Pothenot problem}) was first studied by Snellius, who found a solution around 1615
\cite{Snell,Haasbroek}.
Clearly, here we have a partial case of the aforementioned problem in space, but the points $D$ are taken only from $\Pi^0\setminus \{A, B,C\}$.
\smallskip

Let us recall a classical formula for the volume of a tetrahedron, which will be very useful to us.
If $V$ is the volume of the tetrahedron $ABCD$, then we have
\begin{eqnarray}\label{eq.vol.1}
V^2\!\!\!&=\!\!\! &\frac{1}{36}\, x^2y^2z^2 \Bigl(1 + 2 a_1a_2a_3- a_1^2-a_2^2-a_3^2\Bigr),
\end{eqnarray}
where
$x=d(D,A)$, $y=d(D,B)$, $z=d(D,C)$, $a_1=\cos \angle BDC$, $a_2=\cos \angle ADC$, $a_3=\cos \angle ADB$.
This result was obtained by L.~Euler \cite{Euler},
see also Problems IV and VI in Note 5 (pp. 249, 251) of \cite{Legendre}.

This implies that every point
$(a_1,a_2,a_3)=\bigl(\cos \angle BDC, \cos \angle ADC, \cos \angle ADB \bigr)\in \mathbb{R}^3$
is situated in the set
\begin{equation}\label{eq_pil}
\mathbb{P}=\left\{(x_1,x_2,x_3)\in [-1,1]^3\,|\,1+2x_1x_2x_3-x_1^2-x_2^2-x_3^2 \geq 0 \right\}
\end{equation}
(the ``pillow''), see details e.~g. in \cite[Theorem 1]{NikNik2025}.
We will also often use the surface
\begin{equation}\label{eq.pillowc}
\mathbb{BP}=\left\{(x_1,x_2,x_3)\in [-1,1]^3\,|\,1+2x_1x_2x_3-x_1^2-x_2^2-x_3^2 = 0 \right\}
\end{equation}
(the ``pillowcase''), that is the boundary of $\mathbb{P}$.
\smallskip

It is known that $N(a_1,a_2,a_3)\leq 4$ for any  $(a_1,a_2,a_3)\in \mathbb{P} \setminus \mathbb{BP}$.
This was proved by J.A.~Grunert in \cite{Gru1841}. The set of  $(a_1,a_2,a_3)\in \mathbb{P}$ with $N(a_1,a_2,a_3) \geq 1$ is determined in~\cite{NikNik2025}.
\smallskip

The formula \eqref{eq.vol.1} implies that $D\in \Pi^0\setminus \{A, B,C\}$ if and only if
$(a_1,a_2,a_3)=\bigl(\cos \angle BDC, \cos \angle ADC, \cos \angle ADB \bigr)$
is situated in the set $\mathbb{BP}$.
Therefore, the following problem is also very natural.

\begin{problem} \label{prob.2}
For a given triples $(a_1,a_2,a_3)\in \mathbb{BP}$, determine the number $N(a_1,a_2,a_3)$, i.e.
the number of points $D\in \Pi^0\setminus \{A, B,C\}$, where $\Pi^0$ is the plane $ABC$,
with the property $(a_1,a_2,a_3)=\bigl(\cos \angle BDC, \cos \angle ADC, \cos \angle ADB \bigr)$.
\end{problem}
This article is devoted to a thorough solution to this problem.
\bigskip

In \cite[pp. 203--206]{Snell}, Snellius determined the location of point $\mathbf{o}$ (the roof of his house) using angles taken from point
$\mathbf{o}$ to three known landmarks in Leiden: $\mathbf{y}$ (Pieterskerk), $\mathbf{i}$ (Town Hall), and $\mathbf{u}$ (Hooglandse Kerk).
Note that Pieterskerk (St. Peter's Church) is the oldest church in Leiden,
Town Hall is the historic Leiden Town Hall building on Breestraat square, and
Hooglandse Kerk (Church on the High Land) is a Late Gothic church in Leiden.
\medskip

The mathematician and  geodetic L.~Pothenot, while participating in cartographic projects,
for example, in mapping the Erre River canal commissioned by King Louis XIV,
used a similar method to determine the positions of the required points; see \cite{Pothenot1692}.

\medskip

In 1841, the German mathematician J.A.~Grunert published several papers devoted to solving, among other things,
Pothenot's problem and its extended version, for which he proposed a system of equations that later bore his name.
See \cite{Gru1841.0} and \cite{Gru1841}.

For the history of research into the relevant problems and the methods created to solve them, one can recommend
\cite[P.~165--193]{AwGra}, \cite{HLON}, \cite{Rieck15}, \cite[P.~251--254]{Wild}, \cite{WMMS} and the references therein.
\smallskip

In fact, we deal with the following famous problem in elementary Euclidean geometry:
Given three fixed points, construct a point from which they are seen under three given angles.
This single geometric problem is famous under several names, depending on the field and specific context.

It is called {\sl The Three-Point Problem or Resection Problem} in Surveying and Geodesy,
{\sl Resection} in Navigation (Marine and  Aerial), {\sl Camera Pose Estimation or Perspective-3-Point (P3P)} in Computer Vision and Robotics.
All these demonstrate practical applications of the Snellius--Pothenot problem.

\smallskip

Although the Snellius--Pothenot problem is not mentioned by name in \cite{Yz}, the introduction and study of so-called
{\it antigonal pairs} actually provides a different perspective on  the Snellius--Pothenot problem.
The focus there is on the signed angles, modulo $\pi$, subtended by the triangle vertices at an arbitrary point.
It turns out that an antigonal pair of points is sometimes also a pair of related solution points to the  Snellius--Pothenot problem (for appropriate parameters),
and such solution point pairs always form antigonal pairs. Antigonal pairs are further discussed in \cite{MR}.

\smallskip

Let us consider the following important points in $\mathbb{BP}$:
\begin{eqnarray}\label{eq.spec.points.1} \notag
\widetilde{A}&:=&\Bigl(-\cos \angle BAC,  \,\cos \angle ABC, \, \cos \angle ACB\Bigl),\\
\widetilde{B}&:=&\Bigl(\cos \angle BAC,  \,-\cos \angle ABC,  \,\cos \angle ACB\Bigl),\\\notag
\widetilde{C}&:=&\Bigl(\cos \angle BAC, \, \cos \angle ABC,  \,-\cos \angle ACB\Bigl).
\end{eqnarray}

Let $\mathbb{S}$ denote the circumcircle of triangle $ABC$. The points $A, B, C$ divide $\mathbb{S}$ into three arcs.
It is easy to see that $\bigl(\cos \angle BDC, \cos \angle ADC, \cos \angle ADB \bigr)$ is exactly the point $\widetilde{A}$
(respectively, $\widetilde{B}$ or $\widetilde{C}$), if $D$ is taken from the arc that does not contain the point $A$ (respectively, $B$ or $C$).
This observation shows that finding numbers
$N(a_1,a_2,a_3)$ for points $(a_1,a_2,a_3)\in \mathbb{BP}$
is not a very simple problem. But as will become clear later, for points of  $\mathbb{BP}\setminus\{\widetilde{A}, \widetilde{B}, \widetilde{C}\}$,
the number $N(a_1,a_2,a_3)$ does not exceed $2$.

An important role in the subsequent discussion is played by the planes $x_1=\cos \angle BAC$, $x_2=\cos \angle ABC$, $x_3=\cos \angle ACB$,
that intersect the surface $\mathbb{BP}$ along ellipses  $E_A, E_B, E_C$, respectively, as well as the connected domains of $\mathbb{BP}$
obtained by removing these three ellipses.

Theorem \ref{te.ellipces} gives us the number $N(a_1,a_2,a_3)$ for all points $(a_1,a_2,a_3)$ from the above ellipses.
The number $N(a_1,a_2,a_3)$ on the mentioned connected domains is studied separately for the acute-angled, rectangular, and obtuse-angled cases. Our main results in this direction are contained in Theorems \ref{te.decom_acute}, \ref{te.decom_right}, and \ref{te.decom_obtuse},
which solve Problem~\ref{prob.2} in the case of acute-angled, rectangular, and obtuse-angled base $ABC$, respectively.

\smallskip

\section{Notation and some auxiliary results}

A comprehensive overview of the interrelationships of the elements in a given tetrahedron can be found in \cite{WirDr2009} and \cite{WirDr2014}.
All necessary information about differential-geometric characteristics of smooth curves and surfaces in three-dimensional space can be found e.g. in \cite{Top}.
\smallskip

The names ``pillow'' and ``pillowcase'' for the set $\mathbb{P}$ \eqref{eq_pil} and the surface $\mathbb{BP}$ \eqref{eq.pillowc}
are not standard, they are taken from \cite{NikNik2025}.
It should be noted that $\mathbb{BP}$ is the part of the classical {\it Cayley {\rm(}nodal{\rm)} cubic surface}
in the cube $[-1,1]^3 \subset \mathbb{R}^3$, see e.g. \cite[P.~392]{BuVe}.
Moreover, $\mathbb{BP}$  is also known as a {\it spectrahedron} \cite{Vin}
and also as  {\it Cassini Berlingot} \cite[P.~3]{BoFa}.
\smallskip

In what follows, we will use the following parameterizations and notations:
\begin{equation}\label{eq.Pi+Pi0}
\Pi^+=\{(p,q,r) \in \mathbb{R}^3\,|\, r>0\}, \quad \Pi^0=\{(p,q,0) \in \mathbb{R}^3\},
\end{equation}
\begin{equation}\label{eq.GT}
\GT=\Pi^+ \bigcup \bigl(\Pi^0 \setminus \{A,B,C\} \bigr).
\end{equation}

The latter set is identified with the set of all tetrahedra with the base $ABC$
(we consider only the case when $D \not\in \{A,B,C\}$). We can also formally define the map $F:\GT \to \mathbb{R}^3$:
\begin{equation}\label{eq_f1}
F(D)=\Bigl(\cos \angle BDC,\cos \angle ADC,\cos \angle ADB \Bigr)=(a_1,a_2,a_3).
\end{equation}

As discussed earlier, any point of the image of the function \eqref{eq_f1} belongs to the set~\eqref{eq_pil}.
It is also easy to see that $F(D)\to (1,1,1)$ if and only if $\|D\| \to \infty$.
\smallskip

Let us recall that the base $\triangle ABC$ is assumed fixed.
In particular, its side lengths $d(A,B)=d_3$, $d(B,C)=d_1$, $d(A,C)=d_2$ are fixed.

Let us consider the Grunert system. If we know $a_1=\cos \angle BDC$, $a_2=\cos \angle ADC$, and $a_3=\cos \angle ADB$,
then we have to find the following value: $s_1=d(D,A)$, $s_2=d(D,B)$, and $s_3=d(D,C)$. Using the cosine theorem we get the Grunert system as follows:
\begin{eqnarray}\label{eq_Grunert_1}\notag
s_1^2-2s_1s_2\cdot a_3 +s_2^2 &=& d_3^2,\\
s_1^2-2s_1s_3\cdot a_2 +s_3^2 &=& d_2^2,\\ \notag
s_2^2-2s_2s_3\cdot a_1 +s_3^2 &=& d_1^2.
\end{eqnarray}

\begin{remark}
It is also useful to have a special notation, say, $N^a(a_1,a_2,a_3)$, for the number of points
of points $D\in \Pi^{+}$ satisfying \eqref{eq.numb.1} and
taking into account the multiplicity of these points (see e.~g. \cite[P.~146]{CLO}).
\end{remark}

Various results on the properties of points $F(D)$, where $D\in \Pi^0 \setminus \{A,B,C\}$, can be found in \cite[Section~5]{NikNik2025}.
Let us consider the decomposition
\begin{equation}\label{eq.decom.plane1}
\Pi^0  = \mathbb{B}_{-}\bigcup \mathbb{S} \bigcup \mathbb{B}_{+},
\end{equation}
where $\mathbb{S}$ is the circumscribed circle for the triangle $ABC$, $\mathbb{B}_{-}$ is the open disc bounded by $\mathbb{S}$,
and $\mathbb{B}_{+}$ is the set of points outside $\mathbb{S}$.
\medskip

We denote by $E_A, E_B, E_C$ the ellipses that
are the intersections of the surface $\mathbb{BP}$
with the planes $x=\cos \angle BAC$, $y=\cos \angle ABC$, $z=\cos \angle ACB$, respectively.
Let us recall the following result.

\begin{prop}[\cite{NikNik2025}]\label{pr.limitpoint.1}
The limit points of the points $F(D)$ {\rm(}$D \in \GT${\rm)}, when $D \to A$, $D \to B$, or $D \to C$, are as follows:
$$
\operatorname{Lim}(A)=\left\{\bigl(x_0, y, z\bigr)\in \mathbb{R}^3 \,\left|\,
y^2+z^2-2 x_0yz \leq 1-x_0^2, \, x_0=\cos \angle BAC \right.\right\},
$$
$$
\operatorname{Lim}(B)=\left\{\bigl(x, y_0, z\bigr)\in \mathbb{R}^3 \,\left|\,
x^2+z^2-2y_0xz \leq 1-y_0^2,\, y_0=\cos \angle ABC\right.\right\},
$$
$$
\operatorname{Lim}(C)=\left\{\bigl(x, y,z_0\bigr)\in \mathbb{R}^3 \,\left|\,
x^2+y^2-2z_0xy \leq 1-z_0^2,\, z_0=\cos \angle ACB  \right.\right\}.
$$
\end{prop}
Therefore,  $E_A, E_B, E_C$ are the boundaries of $\operatorname{Lim}(A), \operatorname{Lim}(B), \operatorname{Lim}(C)$ in the corresponding planes, respectively.

\section{On the number of solutions for the Grunert system on $\mathbb{BP}$}\label{se.numersol}

Here we consider a reduction of \eqref{eq_Grunert_1} to a quartic equation with respect to $s_1^2$ according to the method from \cite{KHF}.
The system \eqref{eq_Grunert_1} can be rewritten as follows:

\begin{eqnarray}\label{eq_Grunert_4}
\left(
               \begin{array}{c}
                 s_2^2 \\
                 s_3^2 \\
                 s_2 s_3 \\
               \end{array}
\right)=
\left(
                \begin{array}{ccc}
                2a_3s_1& 0 & d_3^2-s_1^2\\
                0& 2a_2s_1 & d_2^2-s_1^2\\
                \frac{a_3s_1}{a_1}& \frac{a_2s_1}{a_1}& \frac{d_2^2+d_3^2-d_1^2-2s_1^2}{2a_1} \\
               \end{array}
\right)
\left(
               \begin{array}{c}
                 s_2 \\
                 s_3 \\
                 1 \\
               \end{array}
\right).
\end{eqnarray}

Note, that the matrix above depends only on $s_1$,  $a_i$, and $d_i$ for $i=1,2,3$.
Applying~\eqref{eq_Grunert_4} (twice) to both sides of each of the following equalities
$$
s_2^2 s_3=(s_2s_3)s_2, \quad (s_2s_3)s_3=s_3^2 s_2, \quad (s_2s_3)(s_2s_3)=s_2^2s_3^2,
$$
we get some linear equalities  $c_{i,1} s_2 +c_{i,2} s_3 +c_{i,3}=0$, $i=1,2,3$ (all entries $c_{i,j}$ depends only on $s_1$,  $a_i$, and $d_i$ for $i=1,2,3$).
This implies that $\det (c_{i,j}) =0$. This equality is a polynomial equation with respect $s_1$.
More precisely, it has a form $-\frac{1}{16a_1^4} \cdot \widetilde{P} (s_1^2)=0$, where
\begin{equation}\label{eq_polptul1}
\widetilde{P} (u) = \Lambda_4\, u^4 +\Lambda_3\, u^3+ \Lambda_2\, u^2 +\Lambda_1\, u +\Lambda_0,
\end{equation}
{\begin{small}
\begin{eqnarray*}
\Lambda_4 &=& 16(2a_1a_2a_3 - a_1^2 - a_2^2 - a_3^2 + 1)^2 ,\\
\Lambda_3 &=& 32(2a_1a_2a_3 - a_1^2 - a_2^2 - a_3^2 + 1)\bigl((2a_2^2a_3^2 - a_1a_2a_3 - a_2^2 - a_3^2 + 1)d_1^2 \bigr.\\
&& \bigl. - (a_1a_2a_3 - a_1^2 - a_3^2 + 1)d_2^2 - (a_1a_2a_3 - a_1^2 - a_2^2 + 1)d_3^2\bigr),
\\%\end{eqnarray*}
%\begin{eqnarray*}
\Lambda_2 &=& 16(d_2^4 + 4d_2^2d_3^2 + d_3^4)a_1^4 + 16(d_1^2 - d_3^2)^2a_2^4 + 16(d_1^2 - d_2^2)^2a_3^4\\
&& - 32(d_1^2d_2^2 + d_1^2d_3^2 + d_2^4 + 4d_2^2d_3^2 + d_3^4)a_2a_3a_1^3 \\
&&- 8(5d_1^2-d_2^2-5d_3^2)(d_1^2 - d_2^2 - d_3^2)a_2^2 - 8(5d_1^2 - 5d_2^2 - d_3^2)(d_1^2 - d_2^2 - d_3^2)a_3^2\\
&& + 16(3d_1^4 - 4d_1^2d_2^2 - 4d_1^2d_3^2 + d_2^4 +d_3^4)a_2^2a_3^2 + 64(d_1^2d_2^2 + d_1^2d_3^2 + d_2^2d_3^2)a_1^2a_2^2a_3^2 \\
&& + 24(d_1^2 - d_2^2 - d_3^2)^2- 8a_1(1-2a_2^2 - 2a_3^2)(a_1-2a_2a_3)d_1^4 \\
&& + 8a_1^2\bigl((6d_2^2 + 6d_3^2-8a_2^2d_3^2 - 8a_3^2d_2^2)d_1^2 + 2a_2^2d_2^4 + 4a_2^2d_2^2d_3^2 + 6a_2^2d_3^4 \bigr.\\
&& \bigl. + 6a_3^2d_2^4 + 4a_3^2d_2^2d_3^2 + 2a_3^2d_3^4 - 5d_2^4 - 14d_2^2d_3^2 - 5d_3^4\bigr)\\
&&  -16a_1a_2a_3 \bigl((2d_1^2d_2^2 - 4d_1^2d_3^2 + 2d_2^2d_3^2 + 2d_3^4)a_2^2 \bigr.\\
&&\bigl. + (2d_1^2d_3^2 + 2d_2^4 + 2d_2^2d_3^2-4d_1^2d_2^2)a_3^2 + 2d_1^2d_2^2 + 2d_1^2d_3^2 - d_2^4 - 10d_2^2d_3^2 - d_3^4\bigr),
\end{eqnarray*}
\begin{eqnarray*}
\Lambda_1 &=& 8(d_1^2 - d_2^2 - d_3^2)^3 + 8(d_1^2 - d_2^2 - d_3^2)(d_1^2d_2^2 + d_1^2d_3^2 - d_2^4 - 6d_2^2d_3^2 - d_3^4)a_1^2\\
&& - 8(d_1^2 - d_3^2)(d_1^2 - d_2^2 - d_3^2)^2a_2^2 - 8(d_1^2 - d_2^2)(d_1^2 - d_2^2 - d_3^2)^2a_3^2 - 32d_2^2d_3^2(d_2^2 + d_3^2)a_1^4\\
&& + 8(d1^2 - d_2^2 - d_3^2)(d_1^4 - d_2^4 + 6d_2^2d_3^2 - d_3^4)a_1a_2a_3 - 16d_2^2(d_1^4 - 2d_1^2d_2^2 + d_2^4 + d_3^4)a_1^2a_3^2\\
&&  - 16d_3^2(d_1^4 - 2d_1^2d_3^2 + d_2^4 + d_3^4)a_1^2a_2^2 + 32d_2^2d_3^2(d_1^2 + d_2^2 +d_3^2)a_2a_3a_1^3,\\
\Lambda_0 &=&  (2a_1d_2d_3 + d_1^2 - d_2^2 - d_3^2)^2 (2a_1d_2d_3 - d_1^2 + d_2^2 + d_3^2)^2=(4a_1^2d_2^2d_3^2-(d_1^2 - d_2^2 - d_3^2)^2)^2.\\
\end{eqnarray*}
\end{small}

\begin{remark}\label{rem.zero.coeff}
We see that $\Lambda_0=0$ if and only if  $2a_1 d_2d_3=\pm (d_1^2 - d_2^2 - d_3^2)$ or, equivalently, $a_1=\pm \cos \angle BAC$.
\end{remark}

\begin{remark}\label{rem.invar.3}
It should be noted that polynomial $\widetilde{P} (u)$  (see \eqref{eq_polptul1}) are invariant under the following transformations:
$$
(a_1,a_2,a_3) \mapsto (a_1,-a_2,-a_3), \, (a_1,a_2,a_3) \mapsto (-a_1,a_2,-a_3),\, (a_1,a_2,a_3) \mapsto (-a_1,-a_2,a_3).
$$
This means that  $\widetilde{P} (u)$ can be applied not only for the original Grunert system \eqref{eq_Grunert_1},
but also for three similar systems, where exactly two of the numbers $a_1, a_2, a_3$ are taken with the opposite signs.
\end{remark}

\medskip

Note that the polynomial \eqref{eq_polptul1} can be obtained also by the command {\sl EliminationIdeal} in Maple.
For this goal one need to consider the ideal $J_1$ generated by polynomials $U_1= s_1^2 + s_2^2-2a_3s_1s_2 - d_3^2$,
$U_2= s_1^2 + s_3^2-2a_2s_1s_3 - d_2^2$, $U_3= s_2^2 + s_3^2-2a_1s_2s_3 - d_1^2$, and $U_4=u-s_1^2$ and find its elimination ideal with respect to
the variables $a_1,a_2,a_3,d_1,d_2,d_3, u$. This ideal is generated just by the polynomial \eqref{eq_polptul1}.
See also \cite[Equation (2)]{Rieck12}.

\medskip

Obviously, $2a_1a_2a_3 - a_1^2 - a_2^2 - a_3^2 + 1=0$ implies $\Lambda_4=\Lambda_3=0$.
Hence, we have the following helpful corollary from \eqref{eq_polptul1}.

\begin{corollary}\label{cor.2sol}
If $(a_1,a_2,a_3)\in\mathbb{BP}$, i.e.  $2a_1a_2a_3 - a_1^2 - a_2^2 - a_3^2 + 1=0$,  then $u=s_1^2$ satisfies the equation
$$
\Lambda_2\, u^2 +\Lambda_1\, u +\Lambda_0=0.
$$
\end{corollary}

Note, that under the substitution in indices  $1\to 2\to 3\to 1$ for the values $a_i$, $d_i$, and $s_i$,
we easily get similar quadratic equations for $u=s_2^2$ and $u=s_3^2$.

\begin{lemma}\label{le.solpol}
If a point $(a_1,a_2,a_3)\in\mathbb{BP}\setminus\{\widetilde{A}, \widetilde{B}, \widetilde{C}\}$ is such that $\Lambda_2=\Lambda_1=\Lambda_0=0$,
then $(a_1,a_2,a_3)=(-\cos \angle BAC, -\cos \angle ABC,  -\cos \angle ACB)$.
\end{lemma}

\begin{proof}
The equality $\Lambda_0=0$ implies either $a_1=\frac{d_1^2 - d_2^2 - d_3^2}{2d_2d_3}$ or  $a_1=-\frac{d_1^2 - d_2^2 - d_3^2}{2d_2d_3}$.
Substituting this value in the equations $\Lambda_1=0$ and $2a_1a_2a_3 - a_1^2 - a_2^2 - a_3^2 + 1=0$ and eliminating after this $a_2$ and $a_3$ from the obtained system,
we get that $(a_1,a_2,a_3)$ is one of the following points:
{\scriptsize\begin{eqnarray*}
\left(-\frac{d_2^2 + d_3^2 - d_1^2}{2d_2d_3}, \frac{d_1^2 + d_3^2 - d_2^2}{2d_1d_3},\frac{d_1^2 + d_2^2 - d_3^2}{2d_1d_2}  \right),&&
\left(\frac{d_2^2 + d_3^2 - d_1^2}{2d_2d_3}, -\frac{d_1^2 + d_3^2 - d_2^2}{2d_1d_3},\frac{d_1^2 + d_2^2 - d_3^2}{2d_1d_2}  \right),\\
\left(\frac{d_2^2 + d_3^2 - d_1^2}{2d_2d_3}, \frac{d_1^2 + d_3^2 - d_2^2}{2d_1d_3},-\frac{d_1^2 + d_2^2 - d_3^2}{2d_1d_2}  \right),&&
\left(-\frac{d_2^2 + d_3^2 - d_1^2}{2d_2d_3}, -\frac{d_1^2 + d_3^2 - d_2^2}{2d_1d_3},-\frac{d_1^2 + d_2^2 - d_3^2}{2d_1d_2}  \right).
\end{eqnarray*}}
It is easy to verify that $\Lambda_2=0$ for all these points.
The first three points are exactly equal to $\widetilde{A}$, $\widetilde{B}$, and $\widetilde{C}$ by the cosine law. The fourth point is
$(a_1,a_2,a_3)=(-\cos \angle BAC, -\cos \angle ABC, -\cos \angle ACB)$. The lemma is proved.
\end{proof}

\begin{remark} We see that the condition $(a_1,a_2,a_3)\in\mathbb{BP}\setminus\{\widetilde{A}, \widetilde{B}, \widetilde{C}\}$ implies that there are
at most two solutions for $s_1$ in our problem. The same we can say about $s_2$ and $s_3$ of course. Therefore, we can get at most $8$ triples
$(s_1,s_2,s_3)$ that can be solutions of  the  Grunert system.
\end{remark}

In fact we have the following

\begin{prop}\label{pr.spp_sol_1}
The  Grunert system \eqref{eq_Grunert_4} has at most two solutions for any point
$(a_1,a_2,a_3)\in\mathbb{BP}\setminus\{\widetilde{A}, \widetilde{B}, \widetilde{C}\}$.
\end{prop}

\begin{proof}
We can find the general solution for the Grunert system \eqref{eq_Grunert_1}, where
$$
a_1 = \frac{1-t^2}{1+t^2},\quad  a_2 = \frac{1-s^2}{1+s^2}, \quad a_3 = \frac{(1-s^2)(1-t^2) + 4st}{(1+s^2)(1+t^2)}, \qquad t,s \in \mathbb{R}.
$$
This is a rational parametrization of the surface
$\mathbb{BP}$.
By eliminating the variables $s_1, s_2, s_3$ using Gr\"{o}bner bases, we can reduce our system to the following form.

There are $12$ polynomials $P_i=P_i(d_1,d_2,d_3,t,s)$, $1\leq i \leq 12$, such that
$$
P_1\cdot s_1^4+P_2\cdot s_1^2+P_3=0,
$$
$$
P_2^2-4P_1P_3= P_4^2\cdot (d_1 + d_2 + d_3)(d_1 + d_2 - d_3)(d_1  + d_3 - d_2)(d_2 + d_3 - d_1),
$$
$$
P_5\cdot s_1\cdot s_2=P_6 \cdot s_1^2+P_7,
$$
$$
(P_8\cdot s_1^2 +P_9)\cdot s_1\cdot s_3=P_{10}\cdot s_1^4+P_{11}\cdot s_1^2 +P_{12}.
$$
We can write down these polynomial explicitly.
The polynomial $P_1\cdot s_1^4+P_2\cdot s_1^2+P_3$ is obtained in Corollary \ref{cor.2sol} (before we expressed $a_1$, $a_2$, $a_3$ through $t$ and $s$).

If at least one of the coefficients $P_i=P_i(d_1,d_2,d_3,t,s)$, $i=1,2,3$, is non-zero for a given point
$(a_1,a_2,a_3)\in\mathbb{BP}\setminus\{\widetilde{A}, \widetilde{B}, \widetilde{C}\}$, then there are at most $2$ suitable values for $s_1$, while $s_2$ and $s_3$
are determines uniquely by a chosen $s_1$.

Let us suppose that $P_1=P_2=P_3=0$ for a given point $(a_1,a_2,a_3)$.
By Lemma~\ref{le.solpol} we obtain
$(a_1,a_2,a_3)=(-\cos \angle BAC, -\cos \angle ABC,  -\cos \angle ACB)$.
If there is a point $D$ such that  $(a_1,a_2,a_3)=(\cos \angle BDC, \cos \angle ADC,  \cos \angle ADB)$, then
$\angle BDC + \angle ADC + \angle ADB=(\pi-\angle BAC) +(\pi - \angle ABC) +(\pi  - \angle ACB)=3\pi-\pi=2\pi$.
Therefore, $D$ is situated inside the triangle $ABC$. It is easy to see that such a point does exist only if the triangle $ABC$ is acute-angled.
Indeed, it this case $D$
is located on each of the three circles that are obtained from $\mathbb{S}$ by symmetry with respect to the lines $AB$, $AC$, and $BC$.
In other words, the point $D$ must be the orthocenter of the triangle $ABC$ (see e.g. \cite[P.~37]{CoGre}, and since $D$
is inside the triangle, the triangle must be acute-angled.
Therefore, we have exactly one solution to the Grunert system \eqref{eq_Grunert_1}
(i.e.~$N(a_1,a_2,a_3)=1$) for $(a_1,a_2,a_3)=(-\cos \angle BAC, -\cos \angle ABC,  -\cos \angle ACB)$.
This argument completes the proof of the proposition.
\end{proof}

\section{Some important tools}

Let us consider some open subset $\Omega$ in $\mathbb{R}^n$, some $n$-dimensional $C^1$-smooth manifold $M$ and some $C^1$-smooth map
\begin{equation}\label{eq.smothmap}
h: \Omega \rightarrow M.
\end{equation}
For any subset $L\subset M$, $h^{-1}(L)$ denotes the pre-image of $L$ with respect to $h$, i.e $h^{-1}(L)=\{x \in \Omega\,|\, h(x)\in L\}$.

Denote by $\mathcal{D}$ the set of point in  $\Omega$, where $h$ is degenerate, i.e. the Jacobian  of $h$ (in some local coordinates for $M$) is zero:
$$
\mathcal{D}:=\{ x\in \Omega\,|\, |Jh(x)|=0\}.
$$
We need the following key result.

\begin{prop}\label{pr.numsol1}
Suppose that there is  $k \in \mathbb{N}$ such that the number of solutions of the equation $h(x)=b$ is at most $k$ for any $b\in M$.
Let $S$ be a closed connected subset of $M$ such that $h^{-1}(S)$
is a compact subset of $\Omega \setminus \mathcal{D}$. Then
the number of solutions of the equation $h(x)=b$ is the same for all $b\in S$.
\end{prop}

\begin{proof}
Suppose that $k_0$ is the maximum number for solutions of the equation $h(x)=b$ for $b\in S$. Take any point $b_0\in S$  with $k_0$ solutions.
The basic idea of the proof is as follows. For any other $b_1\in S$, we construct a continuous curve connecting $b_0$ to $b_1$ and construct
$k_0$ branches along this curve to solve the equation $h(x)=b$, where $b$ is the point of the above curve. The conditions of the proposition allow this to be done.

For a given $b_0\in S$  with $k_0$ solutions, there are $k_0$ points $x_i \in S$, $i=1,2,\dots, k_0$, such that $h(x_i)=b_0$ for all these points.
Since $h^{-1}(S)\bigcap \mathcal{D}=\emptyset$, there is
a neighborhood $V$ of $b_0$ in $M$ and  neighborhoods $U_i$ of $x_i$ in $\Omega$, $1\leq i \leq k_0$, such that $U_i\bigcap U_j =\emptyset$ for $i\neq j$ and
the restrictions of $h$ to $U_i$ are  diffeomorphisms onto $V$ (this follows from the Inverse Function theorem).
Therefore, for points $b$ sufficiently close to $b_0$, the equation $h(x)=b$ has at least $k_0$ solutions (at least one solution on every $U_i$, $1\leq i \leq k_0$).
But $k_0$ is the maximum possible number of such solutions. Hence, we have exactly $k_0$ solutions for all $b$ sufficiently close to $b_0$.
This means that the set $S(k_0)$ of points $b\in S$ with $k_0$ solutions of the equation $h(x)=b$ is open in $S$ (in the induced topology).

Let us suppose that there is a point $b_1\in  S\setminus S(k_0)$. Without loss of generality we may suppose that $b_1$ is on the boundary
of a connected component $W$ of $S(k_0)$ such that $b_0\in W$. Then we can consider a continuous curve $\varphi: [0,1] \rightarrow S$ such that
$\varphi(t) \in W$ for $t\in [0,1)$, $\varphi(0)=b_0$, $\varphi(1)=b_1$. Then we have $k_0$ pre-images of this curve, restricted on $[0,1)$,
say $\psi_i(t)$, such that $h(\psi_i(t))=\varphi(t)$ for all  $t\in [0,1)$ and $1\leq i \leq k_0$.
For any fixed $i$ the curve $\psi_i(t)$ has at least one limit point $T_i\in h^{-1}(S)$ for $t \to 1$
(recall that $h^{-1}(S)$ is a compact subset of $\Omega \setminus \mathcal{D}$).
Using the limit passage, we obtain that $h(T_i)=b_1$ for all $1\leq i \leq k_0$. Moreover, $T_i \neq T_j$ for $i \neq j$.
Otherwise, the Jacobian  $|Jh(T_i=T_j)|$ is zero (if it is not zero, then in some small neighborhood of $T_i=T_j$, there is only one pre-image
of $\varphi(t)$ for $t$ sufficiently close to $1$, but we have at least two such pre-images: $\psi_i(t)$ and $\psi_j(t)$).
Therefore, $T_i \neq T_j$ for $i \neq j$, and the point $b_1$
has at least $k_0$ solutions of the equation $h(x)=b_1$, which implies $b_1 \in S(k_0)$, that is impossible. Therefore, we get $S=S(k_0)$ that proves the proposition.
\end{proof}

\bigskip
The above-proven proposition has many useful applications. Below, we'll consider one simple application related to the problem we're interested in.
\bigskip

Now we consider a Cartesian coordinate system on the plane such that
$A=(1,0)$, $B=\left(\frac{u^2 - 1}{u^2 + 1}, \frac{2u}{u^2 + 1}\right)$, $C=\left(\frac{v^2 - 1}{v^2 + 1}, \frac{2v}{v^2 + 1}\right)$.

Now we consider a special map
\begin{equation}\label{eq.func_h}
h=(h_1,h_2,h_3):\Omega \rightarrow M=\mathbb{BP}\setminus \{(1,1,1),(-1,-1,1), (-1,1,-1), (1,-1,-1)\},
\end{equation}
where $\Omega=\mathbb{R}^2 \setminus \{A,B,C\}$ with points $D=(x,y)$,
$$
h_1=\cos \angle BDC =\frac{(W_2+2vy-2x)u^2+2(2v-y-v^2y)u+  W_2+2(xv^2+vy-u^2-v^2)}
{W_1^{1/2}W_2^{1/2}(u^2 + 1)^{1/2}(v^2 + 1)^{1/2}},
$$

$$
h_2=\cos \angle ADC =
\frac{(x^2 + y^2 - 2x + 1)v^2 - 2vy + x^2 + y^2 - 1}{W_2^{1/2}W_3^{1/2}(v^2 + 1)^{1/2}},
$$
$$
h_3=\cos \angle ADB =
\frac{(x^2 + y^2 - 2x + 1)u^2 - 2uy + x^2 + y^2 - 1}{W_1^{1/2}W_3^{1/2}(u^2 + 1)^{1/2}},
$$

\begin{eqnarray*}
W_1&=&(x^2 + y^2 - 2x + 1)u^2 - 4uy + x^2 + y^2 + 2x + 1,\\
W_2&=&(x^2 + y^2 - 2x + 1)v^2 - 4vy + x^2 + y^2 + 2x + 1,\\
W_3&=&x^2 + y^2 - 2x + 1.
\end{eqnarray*}

If $D \in \{A,B,C\}$ then $W_i=0$ for some $i\in \{1,2,3\}$.
Hence, we assume that
$D=(x,y)\not \in \left\{A=(1,0), B=\left(\frac{u^2 - 1}{u^2 + 1}, \frac{2u}{u^2 + 1}\right), C=\left(\frac{v^2 - 1}{v^2 + 1}, \frac{2v}{v^2 + 1}\right) \right\}$.

It is easy to check that

$$
\frac{\partial h_1}{\partial x}\frac{\partial h_2}{\partial y}-\frac{\partial h_1}{\partial y}\frac{\partial h_2}{\partial x}=
\frac{16(x^2 + y^2 - 1)(vx - v + y)(uvx - uv + uy + vy - x - 1)(u - v)^2}{W_1^{3/2}W_2^2W_3^{3/2}(u^2 + 1)^{1/2}(v^2 + 1)},
$$
$$
\frac{\partial h_1}{\partial x}\frac{\partial h_3}{\partial y}-\frac{\partial h_1}{\partial y}\frac{\partial h_3}{\partial x}=
\frac{16(x^2 + y^2 - 1)(uvx - uv + uy + vy - x - 1)(u - v)^2(ux - u + y)}{W_1^{2}W_2^{3/2}W_3^{3/2}(u^2 + 1)(v^2 + 1)^{1/2}},
$$
$$
\frac{\partial h_2}{\partial x}\frac{\partial h_3}{\partial y}-\frac{\partial h_2}{\partial y}\frac{\partial h_3}{\partial x}=
\frac{16(x^2 + y^2 - 1)(vx - v + y)(ux - u + y)(u - v)}{W_1^{3/2}W_2^{3/2}W_3^{3/2}(u^2 + 1)^{1/2}(v^2 + 1)^{1/2}}.
$$
This implies that the vectors $\left(\frac{\partial h_1}{\partial x},\frac{\partial h_2}{\partial x},\frac{\partial h_3}{\partial x}\right)$ and
$\left(\frac{\partial h_1}{\partial y},\frac{\partial h_2}{\partial y},\frac{\partial h_3}{\partial y}\right)$ are linearly dependent
if and only if $x^2+y^2=1$. Hence, we have proved the following result.

\begin{lemma}\label{le.degen1}
The above map $h=(h_1,h_2,h_3):\Omega \rightarrow M$ is $C^1$-smooth in
$\mathbb{R}^2\setminus \left\{A, B, C \right\}$
and is non-degenerate in all points of the set $\mathbb{R}^2\setminus \mathbb{S}$, where $\mathbb{S}=\{(x,y)\in \mathbb{R}^2\,|\, x^2+y^2=1\}$.
\end{lemma}

\medskip

We recall the notations $x_0:=\cos \angle BAC$, $y_0:=\cos \angle ABC$, $z_0:=\cos \angle ACB$.
The point $(x_0,y_0,z_0)$ is a common point of three planes containing
the limit solid ellipses $\operatorname{Lim}(A)$, $\operatorname{Lim}(B)$, and $\operatorname{Lim}(C)$
(the equations of these planes are  $x=x_0$, $y=y_0$, $z=z_0$, respectively; see Proposition \ref{pr.limitpoint.1}).
Moreover, $(x_0,y_0,z_0)$ lies in the interior of $\mathbb{P}$, on the surface $\mathbb{BP}$, and outside of $\mathbb{P}$ in accordance with the property of
$\triangle ABC$ to be  acute-angled, right-angled and obtuse-angled respectively (which is equivalent to $x_0$ >0, $x_0=0$, $x_0<0$, respectively, if $y_0>0$ and $z_0>0$).
It is easy to see from the equality $1+2xyz-x^2-y^2-z^2=4xyz$ for
$(x,y,z)=(\cos \angle BAC, \cos \angle ABC, \cos \angle ACB)$.
\smallskip

Let us recall some important properties of the set $\mathbb{P}$ \eqref{eq_pil} and its boundary $\mathbb{BP}$ \eqref{eq.pillowc}, see e.g. \cite[Proposition~3]{NikNik2025}.
It is known that $\mathbb{P}$ (the pillow) is a convex subset of the cube $[-1,1]^3 \subset \mathbb{R}^3$ containing a tetrahedron $TP$ with the vertices
 $(1,1,1)$, $(1,-1,-1)$, $(-1,1,-1)$, and $(-1,-1,1)$. Moreover, the surface $\mathbb{BP}$ (the boundary of  $\mathbb{P}$) is smooth except four vertices of
$TP$ and has positive Gaussian curvature
in all points except the edges of the tetrahedron $TP$.
This implies the following property: any segment with non-zero length in $\mathbb{BP}$
is a part of some edge of the tetrahedra $TP$ (with the vertices  $(1,1,1)$, $(1,-1,-1)$, $(-1,1,-1)$, and $(-1,-1,1)$).
\smallskip

Now, we consider $8$ open connected subsets of
\begin{equation}\label{eq.nd}
ND:=\left\{ (x,y,z)\in \mathbb{R}^3\,|\, x\neq x_0, y\neq y_0, z\neq z_0 \right\}.
\end{equation}
We define $\Omega(i,j,k)$ to be the set of points $(x,y,z)$ from $ND$ such that $i=\sgn(x-x_0)$, $j=\sgn(y-y_0)$, $k=\sgn(z-z_0)$.
Obviously, the boundary of any of these 8 sets is part of the union of three planes containing the limit ellipses $\operatorname{Lim}(A)$,
$\operatorname{Lim}(B)$, and $\operatorname{Lim}(C)$ (any of these sets is located in its own octant).

Now we consider the $C^1$-smooth surface
\begin{equation}\label{eq.surf_M}
M=\mathbb{BP}\setminus \{(1,1,1),(-1,-1,1), (-1,1,-1), (1,-1,-1)\}
\end{equation}
and its 8 open subsets $\mathbb{BP}(i,j,k):=M\bigcap \Omega(i,j,k)$ for various vales of $i,j,k$.
It is clear that the points $(1,1,1)$, $(-1,-1,1)$, $(-1,1,-1)$, $(1,-1,-1)$ are in the closures of $\mathbb{BP}(+,+,+)$, $\mathbb{BP}(-,-,+)$, $\mathbb{BP}(-,+,-)$, and
$\mathbb{BP}(+,-,-)$, respectively.
\medskip

It should be noted that $\mathbb{BP}(i,j,k)=\emptyset$ for some $i,j,k$ and some bases $ABC$, see Lemma \ref{le.obtusen} below.
On the other hand, $\mathbb{BP}(i,j,k)$ should not be connected in general. In any case, we have the following result.

\begin{lemma}\label{le.noncon1}
Any $\mathbb{BP}(i,j,k)$ has at most two connected components.
\end{lemma}

\begin{proof} Let us consider some non-connected $\mathbb{BP}(i,j,k)$. Let $l$ be any ray through $\Omega(i,j,k)$ issuing from the point $(x_0,y_0,z_0)$.

If $l$ intersects the interior of $\mathbb{P}$, then $l \bigcap\mathbb{P}$ is a segment and
$l \bigcap\mathbb{BP}$ is a two-point set $\{L',L''\}$, where $L'$ is closer to $(x_0,y_0,z_0)$ than $L''$.
The union of all such points $L'$ (as well as of all points $L''$) is an open connected region in $\mathbb{BP}$
(it is important that the union is taken on all rays $l$ containing some points of the interior of $\mathbb{P}$).

Let us suppose that there is a ray $\tilde{l}$ such that $\tilde{l} \bigcap \mathbb{BP} \neq \emptyset$ and
$\tilde{l}$ does not intersect the interior of $\mathbb{P}$. In this case $\tilde{l} \bigcap \mathbb{BP}$ is a one-point set.
Indeed, $ABC$ is non-degenerate, hence $|s|<1$ for all $s \in \{x_0,y_0,z_0\}$, but any segment in $\mathbb{BP}$
is a part of some edge of the tetrahedra with the vertices  $(1,1,1)$, $(1,-1,-1)$, $(-1,1,-1)$, and $(-1,-1,1)$ (see discussion above).
In this case, $\mathbb{BP}(i,j,k)$ has only one component
(the two components with points of types $L'$ and $L''$ are glued together at the point $\tilde{l}$ and at similar points of other rays that intersect
$\mathbb{P}$ at exactly some points of $\mathbb{BP}$).

If such a ray $\tilde{l}$ does not exist, then
$\mathbb{BP}(i,j,k)$ has exactly two connected components (one component with all points $L'$ and other with all points $L''$).
\end{proof}
\smallskip

\begin{remark} It is easy to prove the following result:
If $\mathbb{BP}(i,j,k)$ has two connected components, then the set $\mathbb{BP}(-i,-j,-k)$ is empty.
\end{remark}

In what follows, if $\mathbb{BP}(i,j,k)$ is not connected, then $\mathbb{BP}(i,j,k)'$ and $\mathbb{BP}(i,j,k)''$ are its component containing
the points $L'$ and $L''$ (from the proof of Lemma \ref{le.noncon1}), respectively. It is mean that $\mathbb{BP}(i,j,k)'$ is closer to the point
$(x_0,y_0,z_0)$ than $\mathbb{BP}(i,j,k)''$.
\medskip

For a given $\mathbb{BP}(i,j,k)$ we denote by $\partial \mathbb{BP}(i,j,k)$ its boundary in  $\mathbb{BP}$.
If $\mathbb{BP}(i,j,k)$ is connected, then $\partial \mathbb{BP}(i,j,k)$ contains a union of three elliptic arcs
(of the ellipses $E_A$, $E_B$, $E_C$) and (in addition) at most one
of the points $(1,1,1)$, $(-1,-1,1)$, $(-1,1,-1)$, $(1,-1,-1)$. If $\mathbb{BP}(i,j,k)$ is not connected, then we consider its two components
$\mathbb{BP}(i,j,k)'$ and $\mathbb{BP}(i,j,k)''$ separately.

Let us consider $U$, any connected component of some $\mathbb{BP}(i,j,k)$.
Now for any $\varepsilon >0$ we consider
\begin{equation}\label{eq.eps1}
U^{\varepsilon}=\left\{(x,y,z) \in U\,|\,d\Bigl((x,y,z),\partial U\Bigr)\leq \varepsilon \right\},
\end{equation}
where $d$ is the Euclidean distance in $\mathbb{R}^3$. It is clear that $U^{\varepsilon}$ is a compact and connected subset in $U$ and
$$
U=\bigcup_{\varepsilon >0}\,\, U^{\varepsilon}.
$$

\begin{lemma}\label{le.numsol1}
For any connected component $U$ of a given $\mathbb{BP}(i,j,k)$ and any $\varepsilon >0$,
the number of solutions of the system $h(x,y)=b$ {\bf(}see \eqref{eq.func_h}{\bf)} is constant for all $b \in U^{\varepsilon}$.
\end{lemma}

\begin{proof}
By Lemma \ref{le.degen1}, degenerate points of $h$ form the set $\mathcal{D}=\{ x\in \mathbb{R}^2\,|\, |Jh(x)|= 0\}=\mathbb{S}$.
According to the conditions of Proposition \ref{pr.numsol1}, it suffices to show that the preimage $V:=h^{-1}\bigl(U^{\varepsilon}\bigr)$ is a compact
subset in  $\mathbb{R}^2\setminus \mathbb{S}$.

First, if $V\bigcap \mathbb{S} \neq \emptyset$, then there is $w\in V\bigcap S$, and so $h(w)\in   U^{\varepsilon}$.
In this case $w \in \mathbb{S}\setminus \{A,B,C\}$, hence, $h(w) \in \{\widetilde{A}, \widetilde{B}, \widetilde{C}\}$ which is impossible.
Indeed, all these three points are on the union  $E_A \bigcup E_B \bigcup E_C$ of three ellipses and cannot be situated in $U^{\varepsilon}$ for positive
$\varepsilon$. Therefore, $V \subset \mathbb{R}^2\setminus \mathbb{S}$. It is clear that $V$ is closed. Hence, if $V$ is not compact, then there is a sequence
$\{w_n\}_{n\in \mathbb{N}}$, $w_n \in V$, such that $\|w_n\| \to \infty$ when $n \to \infty$. But in this case
$h(w_n) \to(1,1,1)\in \partial \mathbb{BP}(+,+,+)$, which contradict to the inclusion $h(w_n)\in  U^{\varepsilon}$ for all $n$.
The lemma is proved.
\end{proof}
\smallskip

Since $U=\bigcup_{\varepsilon >0}\,\, U^{\varepsilon}$, we immediately obtain the following important result.

\begin{corollary}\label{co.numsol1}
For any connected component $U$ of $\mathbb{BP}(i,j,k)$, the number of solutions of the system $h(x,y)=b$ is constant for all $b \in U$.
\end{corollary}

It should be noted that the system $h(x,y)=b$  (see \eqref{eq.func_h}) is the Grunert system restricted to the set $\Pi^0 \setminus \{A,B,C\}$.
On the other hand, we know that any solution of Grunert system for a point $b \in \mathbb{BP}$ is from $\Pi^0 \setminus \{A,B,C\}$.

\section{The number of solutions on the ellipses $E_A$, $E_B$, and $E_C$}

We know that $\widetilde{A}, \widetilde{B} \in E_C$, $\widetilde{A}, \widetilde{C} \in E_B$, and
$\widetilde{B}, \widetilde{C} \in E_A$, see \eqref{eq.spec.points.1}.
Note that $\widetilde{A}\in E_A$ if and only if $\angle BAC =\pi/2$.

Recall that the solutions of the Grunert system \eqref{eq_Grunert_1} for $(a_1,a_2,a_3) \in \{\widetilde{A}, \widetilde{B}, \widetilde{C}\}$
consist of an arc of the circle $\mathbb{S}$.

It is useful to consider  points with the following properties: $\widehat{A}\in E_B \bigcap E_C \setminus \{\widetilde{A}\}$,
$\widehat{B}\in E_A \bigcap E_C \setminus \{\widetilde{B}\}$, $\widehat{C}\in E_A \bigcap E_B \setminus \{\widetilde{C}\}$.
Since the intersection of any two such ellipses contains exactly two point (for a non-degenerate base), the above property determines the points
$\widehat{A}$, $\widehat{B}$, and $\widehat{C}$.

It is easy to check (using \eqref{eq.pillowc}) that
\begin{eqnarray}\label{eq.spec.points.2} \notag
\widehat{A}&=&\Bigl(\cos |\angle ABC - \angle ACB|, \,\cos \angle ABC,  \,\cos \angle ACB\Bigl),\\
\widehat{B}&=&\Bigl(\cos \angle BAC,  \,\cos |\angle BAC -\angle ACB |,  \,\cos \angle ACB\Bigl),\\\notag
\widehat{C}&=&\Bigl(\cos \angle BAC,  \,\cos \angle ABC,  \,\cos |\angle BAC - \angle ABC|\Bigl).
\end{eqnarray}

\begin{lemma}\label{le.numsorel0}
We have $\left\{\widetilde{A}, \widetilde{B}, \widetilde{C} \right\}\bigcap \left\{\widehat{A}, \widehat{B}, \widehat{C} \right\}\neq \emptyset$
if and only if the triangle $ABC$ is rectangular. If, say, $\angle BAC =\pi/2$, then $\widehat{B}= \widehat{C}=\widetilde{A}$
\end{lemma}

\begin{proof} If, say, $\angle BAC =\pi/2$ and $\angle ABC =\psi$, then $\widetilde{A}=(0, \cos \psi, \sin \psi)$,
$\widehat{A}=(\sin 2\psi, \cos \psi, \sin \psi)$, $\widetilde{B}=(0, -\cos \psi, \sin \psi)$,
$\widetilde{C}=(0, \cos \psi, -\sin \psi)$, and $\widehat{B}=\widehat{C}=(0, \cos \psi, \sin \psi)=\widetilde{A}$.
On the other hand, if, say, $\widehat{B}=\widetilde{A}$, then $\widetilde{A} \in E_A$ and $\cos \angle BAC=-\cos \angle BAC$, i.e. $\angle BAC= \pi/2$.
\end{proof}

\begin{lemma}\label{le.numsorel1}
If $\angle BAC =\pi/2$, then for any $(a_1,a_2,a_3) \in E_A\setminus \{\widetilde{A}, \widetilde{B}, \widetilde{C}\}$,
the Grunert system \eqref{eq_Grunert_1} has no solution.
\end{lemma}

\begin{proof}
Any such solutions should be from $\mathbb{S}\setminus \{A,B,C\}$, but this is impossible (the images of such points are one of the points
$\widetilde{A}$, $\widetilde{B}$, $\widetilde{C}$).
\end{proof}
\medskip

Now, {\it we consider the case $\angle BAC \neq \pi/2$}. We assume that $A,B,C \in \mathbb{S}$. Let $\gamma_1$ be an open arc of $\mathbb{S}$ with the endpoints
$B$ and $C$, that contains the point $A$. Consider $\gamma_2$, the image of $\gamma_1$ under the reflection with respect to the line $BC$.
Then $\gamma_1 \bigcup \gamma_2$ is the set of points $D \in \mathbb{R}^2\setminus \{B,C\}$ such that $\angle BDC = \angle BAC$.
The image of $\gamma_1 \setminus \{A\}$ under the map $h$ (see \eqref{eq.func_h}) is a set $\{\widetilde{B}, \widetilde{C}\}$.
Let us consider the curve $h(\gamma_2) \subset E_A$. The following proposition is almost obvious.

\begin{prop}\label{pr.numsorel2}
If, in the above conditions,  $D \in \gamma_2$ tend to $B$ {\rm(}respectively, to $C${\rm)}, then $h(D)$ tend to $\widehat{C}$ (respectively, to $\widehat{B}$).
The restriction of $h$ to $\gamma_2$ is injective. The image $h(\gamma_2)$ does not contain the points $\widetilde{B}$ and $\widetilde{C}$.
\end{prop}

\begin{proof}
The first statement is almost obvious and can be checked without any problem. Now, since $\angle BAC \neq \pi/2$, the curve $\gamma_2$
has empty intersection with $\mathbb{S}$. Therefore, in a neighborhood of any point $D \in \gamma_2$ the map $h$ is injective
(it follows from the fact that $h$ at the point $D$ is non-degenerate by Lemma \ref{le.degen1}). It implies that restriction of $h$ to $\gamma_2$ is injective.
The last statement follows from the fact that $h^{-1}\bigl(\widetilde{B}\bigr)\bigcup h^{-1}\bigl(\widetilde{C}\bigr)=\gamma_1 \setminus \{A\}$.
\end{proof}

Let $\theta_A$ is an open arc of the ellipse $E_A$ between the points $\widehat{B}$ and $\widehat{C}$, that does not contain
the points $\widetilde{B}$ and $\widetilde{C}$. The above proposition implies the following corollary.

\begin{corollary}\label{co.numsorel3}
If $\angle BAC \neq \pi/2$, then the Grunert system \eqref{eq_Grunert_1} has exactly one solution for any $(a_1,a_2,a_3) \in \theta_A$
and has no solution for all
$(a_1,a_2,a_3) \in E_A\setminus \left( \theta_A \bigcup \{\widetilde{B}, \widetilde{C}\}\right)$.
\end{corollary}

Lemma \ref{le.numsorel1} and Corollary \ref{co.numsorel3} obviously imply the following result.

\begin{theorem}\label{te.ellipces}
If $\angle BAC =\pi/2$, then for any $(a_1,a_2,a_3) \in E_A\setminus \{\widetilde{A}, \widetilde{B}, \widetilde{C}\}$,
the Grunert system \eqref{eq_Grunert_1} has no solution.
If $\angle BAC \neq \pi/2$, then the Grunert system \eqref{eq_Grunert_1} has exactly one solution for any $(a_1,a_2,a_3) \in \theta_A$
and has no solution for all
$(a_1,a_2,a_3) \in E_A\setminus \left( \theta_A \bigcup \{\widetilde{B}, \widetilde{C}\}\right)$,
where $\theta_A$ is an open arc of the ellipse $E_A$ between the points $\widehat{B}$ and $\widehat{C}$, that does not contain
the points $\widetilde{B}$ and $\widetilde{C}$.

Similar results are true for the angle $\angle ABC$ and the ellipse $E_B$, as well as for the angle $\angle ACB$ and the ellipse $E_C$.
\end{theorem}

\section{Computing the number of solutions on open regions of $\mathbb{BP}$}

If $\mathbb{BP}(i,j,k)$ is connected, then we consider $N(i,j,k)$, the number of solutions of the equation $h(x,y)=b$, where $b \in \mathbb{BP}(i,j,k)$.
To find this number, it suffices to find all solutions of this equation for any $b \in \mathbb{BP}(i,j,k)$. Proposition \ref{pr.spp_sol_1} implies that $N(i,j,k) \leq 2$
for all indices.
It should be noted that this value can depend significantly on whether the base is acute-angled or not (that is, it is rectangular or obtuse-angled).

\begin{remark}\label{rem.figures}
Undoubtedly, visualization of the corresponding results for the cases of acute-angled, rectangular and obtuse-angled bases will help to better understand them.
In Figures
\ref{pillowcase_acute}, \ref{pillowcase_right}, \ref{pillowcase_obtuse} (see below), the points $b \in \mathbb{BP}$ for which the equation $h(x,y)=b$ has two, one, or no solutions are shown in blue, brown, or green, respectively.
\end{remark}

If $\mathbb{BP}(i,j,k)$ is not connected (hence it has two connected component by Lemma~\ref{le.noncon1}),
then we consider $N(i,j,k)'$ and $N(i,j,k)''$, the number of solutions of the equation $h(x,y)=b$,
where $b \in \mathbb{BP}(i,j,k)'$ and $b \in \mathbb{BP}(i,j,k)''$, respectively.
\smallskip

It is easy to see that $\mathbb{BP}(+,-,-)$, $\mathbb{BP}(-,+,-)$, and $\mathbb{BP}(-,-,+)$ are connected for any non-degenerate base $ABC$.

\begin{lemma}\label{le.one.horn}
We have $N(+,-,-)=N(-,+,-)=N(-,-,+)=0$ for any non-degenerate base $ABC$.
\end{lemma}

\begin{proof}
Let us prove $N(+,-,-)=0$, the proof of the other two equalities is similar.
If $N(+,-,-)>0$, then for any point  $b=(b_1,b_2,b_3)\in \mathbb{BP}(+,-,-)$, there is a solution of the equation $h(x,y)=b$.
If we take $b_2$ and $b_3$ sufficiently close to $-1$, then the intersection of the set
$T_{A,C}^{b_2}:=\left\{D\in \Pi^0 \setminus \{A,B,C\}\,|\, \cos \angle ADC =b_2\right\}$
and the set
$T_{A,B}^{b_3}:=\left\{D\in \Pi^0 \setminus \{A,B,C\}\,|\, \cos \angle ADB =b_3\right\}$
will be empty. Therefore, we cannot find a suitable $D=(x,y)$. This contradiction proves the lemma.
\end{proof}
\smallskip

The following result follows from the fact that the point $(x_0,y_0,z_0)$ is in the interior of $\mathbb{P}$ for any acute-angled base.

\begin{lemma}\label{le.acute}
If the base $ABC$ is acute-angled, then all sets $\mathbb{BP}(i,j,k)$ are non-empty and connected.
\end{lemma}

\begin{lemma}\label{le.obtuse}
If $\angle BAC \geq \pi/2$, then  $N(-,-,+)=N(-,+,-)=N(-,+,+)=0$.
\end{lemma}

\begin{proof}
If $h(D)\in \mathbb{BP}(-,j,k)$ for some $j, k \in \{-,+\}$, then $\cos \angle BDC <\cos \angle BAC=x_0\leq 0$, hence,
$D$ is inside the circle $\mathbb{S}$. Since $\angle ABC <\pi/2<\pi - \angle ABC$ and $\angle ACB <\pi/2<\pi - \angle ACB$,
then it is easy to see that $\angle ADC > \angle ABC$ and
$\angle ADB > \angle ACB$, that imply $\cos \angle ADC < \cos \angle ABC=y_0$ and
$\cos \angle ADB <\cos \angle ACB=z_0$, that imply $(j,k)=(-,-)$.
\end{proof}

\smallskip

\begin{lemma}\label{le.obtusen}
If $\angle BAC \geq \pi/2$, then $\mathbb{BP}(-,+,+)=\emptyset$. If $\angle BAC > \pi/2$, then $\mathbb{BP}(+,-,-)$ has two connected components.
\end{lemma}

\begin{proof}
Since $\angle BAC \geq \pi/2$, then $\angle ABC +\angle ACB \leq \pi/2$, $x_0=\cos \angle BAC \leq 0$, $y_0=\cos \angle ABC \in (0,1)$, and
$z_0=\cos \angle ACB \in (0,1)$.
Let us consider the function $g:[0,1]\times[0,1] \rightarrow \mathbb{R}$, $g(y,z)=yz-\sqrt{(1-y^2)(1-z^2)}$. It is easy to see that
$g'_y(y,z)=z+y\sqrt{\frac{1-z^2}{1-y^2}}>0$ and $g'_z(y,z)=y+z\sqrt{\frac{1-y^2}{1-z^2}}>0$ for all $y,z \in (0,1)$. This implies that the minimum value of $g(y,z)$ on
the rectangle $[y_0,1]\times [z_0,1]$  is achieved at the point $(y_0,z_0)$ and it is equal to
$$
y_0z_0-\sqrt{(1-y_0^2)(1-z_0^2)}=\cos (\angle ABC + \angle ACB)=-\cos(\angle BAC)=-x_0.
$$
Therefore, $g(y,z) > -x_0$ for all $(y,z) \in (y_0,1]\times (z_0,1]$. Since the points $(x,y,z)\in \mathbb{BP}$ are such that
$x=yz\pm \sqrt{(1-y^2)(1-z^2)}$, then $\mathbb{BP}$ has empty intersection with  the set  $[-1,-x_0)\times(y_0,1]\times (z_0,1]$ and, consequently, with the set
$[-\infty,-x_0)\times(y_0,\infty]\times (z_0,\infty]$. Since $x_0 \leq 0$, we get $\mathbb{BP}(-,+,+)= \emptyset$.

The last statement can be easily checked.
\end{proof}

In what follows, we will have to find solutions of the Grunert system \eqref{eq_Grunert_1} for some specific points $(a_1,a_2,a_3)$
on the surface $\mathbb{BP}$. Let us recall the method for solving this system, which dates back to~\cite{Gru1841}, see also \cite{HLON}.

We consider the ratio $u=s_2/s_1 > 0$ (recall that $s_1,s_2,s_3 >0$).
It is easy to see that
\begin{eqnarray*}
s_1 = \frac{d_3}{\sqrt{1 - 2 a_3 u + u^2}},\quad s_2 = u s_1 = \frac{u d_3}{\sqrt{1 - 2 a_3 u + u^2}},\\
s_3 =\frac{(u^2 - 1)d_3^2 - (d_1^2 - d_2^2)(1-2a_3u +u^2)}{2(a_1 u - a_2)d_3 \sqrt{1-2a_3u +u^2}}.
\end{eqnarray*}

Moreover, $u$ satisfies the following equation (the Grunert equation):
\begin{equation}\label{eq.Grunert.eq1}
P(u):=P_4 u^4 + P_3 u^3 + P_2 u^2 + P_1 u + P_0 = 0,
\end{equation}
where the coefficients are as follows:
\begin{eqnarray*}
P_4 \!\!&=&\!\! (d_1^2 - d_2^2 - d_3^2)^2 - 4a_1^2d_2^2d_3^2, \\
P_3 \!\!&=&\!\! 8a_1^2a_3d_2^2d_3^2 + 4a_1a_2d_3^2(d_1^2 + d_2^2 - d_3^2) - 4(d_1^2 - d_2^2)(d_1^2 - d_2^2 - d_3^2)a_3, \\
P_2 \!\!&=&\!\!  (4a_3^2 + 2)(d_1^2 - d_2^2)^2 + 2(2a_1^2 + 2a_2^2 - 1)d_3^4-4(a_1^2d_2^2 + a_2^2d_1^2)d_3^2 - 8a_1a_2a_3(d_1^2 + d_2^2)d_3^2, \\
P_1 \!\!&=&\!\! 8a_2^2a_3d_1^2d_3^2 + 4a_1a_2d_3^2(d_1^2 + d_2^2 - d_3^2) - 4(d_2^2 - d_1^2)(d_2^2 - d_1^2 - d_3^2)a_3, \\
P_0 \!\!&=&\!\! (d_2^2 - d_1^2 - d_3^2)^2 - 4a_2^2d_1^2d_3^2.
\end{eqnarray*}
\smallskip
Recall that we are only interested in solutions such that $s_i>0$, $i=1,2,3$.
It is clear that there are at most four such solutions.

\subsection{The stability of the number of solutions}

Let us consider a continuous family of the bases $A_{\mu}B_{\mu}C_{\mu}$, $\mu \in [0,1]$.
We put
$x_0^{\mu}:=\cos \angle B_{\mu}A_{\mu}C_{\mu}$, $y_0^{\mu}:=\cos \angle A_{\mu}B_{\mu}C_{\mu}$, $z_0^{\mu}:=\cos \angle A_{\mu}C_{\mu}B_{\mu}$.
Without loss of generality, way may assume that $A_{\mu}, B_{\mu}, C_{\mu} \in \mathbb{S}$ and even  $A_{\mu}=(1,0)$ for all $\mu \in [0,1]$.

The point $(x_0^{\mu},y_0^{\mu},z_0^{\mu})$ is a common point of three planes containing
the limit solid ellipses $\operatorname{Lim}(A_{\mu})$, $\operatorname{Lim}(B_{\mu})$, and $\operatorname{Lim}(C_{\mu})$
(the equations of these planes are  $x=x_0^{\mu}$, $y=y_0^{\mu}$, $z=z_0^{\mu}$).
\smallskip

For $i,j,k \in \{-,+\}$ and a given  $\mu \in [0,1]$,
we define $\Omega_{\mu}(i,j,k)$ to be the set of points $(x,y,z)$ from $\mathbb{R}^3$ such that $i=\sgn(x-x_0^{\mu})$, $j=\sgn(y-y_0^{\mu})$, $k=\sgn(z-z_0^{\mu})$.

For any $\mu \in [0,1]$, we have
8 open subsets $\mathbb{BP}_{\mu}(i,j,k):=M\bigcap \Omega_{\mu}(i,j,k)$ for various vales of $i,j,k$, see \eqref{eq.surf_M}.

If all $\mathbb{BP}_{\mu}(i,j,k)$ are connected, we consider
$N_{\mu}(i,j,k)$, the number of solutions of the equation $h_{\mu}(x,y)=b$, where $b \in \mathbb{BP}_{\mu}(i,j,k)$
(the function $h_{\mu}$ (see \eqref{eq.func_h}) are constructed according to the base
$A_{\mu}B_{\mu}C_{\mu}$).
If every $\mathbb{BP}_{\mu}(i,j,k)$ has two connected components, we consider similar characteristics for any (continuous by $\mu$) family of components.

\begin{lemma}\label{le.consolgs1}
If a point $\breve{b}\in M$ is such that $\breve{b}\in \Omega_{\mu}(i,j,k)$ for some fixed $i,j,k \in \{-,+\}$ and for all $\mu \in [0,1]$,
then $N_{\mu}(i,j,k)$ is constant for all $\mu \in [0,1]$.
\end{lemma}

\begin{proof}
The functions $h_{\mu}$ are continuous with respect to $\mu$, hence, the number of solutions of the equation $h_{\mu}(x,y)=\breve{b}$ is a constant for all $\mu \in [0,1]$.
Now, it suffices to apply Corollary \ref{co.numsol1}.
\end{proof}

\begin{lemma}\label{le.consolgs2}
If bases $A_0B_0C_0$, $A_1B_1C_1$ and a point $\breve{b}=(\breve{x},\breve{y},\breve{z})\in M$ are such that
\begin{eqnarray*}
\sgn(\breve{x}-\cos \angle B_0A_0C_0)=\sgn(\breve{x}-\cos \angle B_1A_1C_1)=:i,\\
\sgn(\breve{y}-\cos \angle A_0B_0C_0)=\sgn(\breve{y}-\cos \angle A_1B_1C_1)=:j,\\
\sgn(\breve{z}-\cos \angle A_0C_0B_0)=\sgn(\breve{z}-\cos \angle A_1C_1B_1)=:k,
\end{eqnarray*}
then $N_0(i,j,k)=N_1(i,j,k)$.
\end{lemma}

\begin{proof}
It is easy to find a continuous family of bases $A_{\mu}B_{\mu}C_{\mu}$, $\mu \in [0,1]$, such that for $\mu=0$ and $\mu=1$ get bases in the statement of the lemma and
\begin{eqnarray*}
\sgn(\breve{x}-\cos \angle B_{\mu}A_{\mu}C_{\mu})=i,\quad
\sgn(\breve{y}-\cos \angle A_{\mu}B_{\mu}C_{\mu})=j,\quad
\sgn(\breve{z}-\cos \angle A_{\mu}C_{\mu}B_{\mu})=k
\end{eqnarray*}
for all $\mu \in [0,1]$. Now, it suffices to apply Lemma \ref{le.consolgs1}.
\end{proof}

\subsection{The base is acute-angled}

\begin{prop}\label{pr.consolgs3}
If bases $A_0B_0C_0$ and $A_1B_1C_1$ are acute-angled, then $N_0(i,j,k)=N_1(i,j,k)$ for any $i,j,k \in \{-,+\}$.
\end{prop}

\begin{proof} It suffices to consider the case when $A_0B_0C_0$ is an arbitrary acute-angled base and the base $A_1B_1C_1$ is regular.
We can find 1-parameter continuous family $A_{\mu}B_{\mu}C_{\mu}$, $\mu \in [0,1]$ of bases,
such that the difference between the maximal and the minimal angles of the triangle
$A_{\mu}B_{\mu}C_{\mu}$
decreases monotonically to zero. For rather small $|\mu_1-\mu_2|$ in the corresponding 1-parameter family $A_{\mu}B_{\mu}C_{\mu}$, $\mu \in [0,1]$,
we have $N_{\mu_1}(i,j,k)=N_{\mu_2}(i,j,k)$ by Lemma  \ref{le.consolgs2}. Then standard arguments (on the compactness of $[0,1]$) imply the proposition.
\end{proof}
\smallskip

\begin{remark}
The Proposition \ref{pr.consolgs3} does not hold if $A_0B_0C_0$ is obtuse-angled and $A_1B_1C_1$ is acute-angled. In this case the sets
$h_{\mu}^{-1}(x_0^{\mu})\setminus \mathbb{S}$ should be significantly changed when $\mu$ goes from $0$ to $1$. Indeed, $h_{0}^{-1}(x_0^{0})\setminus \mathbb{S}$ is inside
$\mathbb{S}$ but $h_{1}^{-1}(x_0^{1})\setminus \mathbb{S}$ is outside
$\mathbb{S}$. Since points of $\mathbb{S}$ could be degenerate for $h_{\mu}$ (see the proof of Lemma \ref{le.degen1}), then we can not state that the number
$N_{\mu}(i,j,k)$ is constant with respect to $\mu$.
\end{remark}

The ``pillowcase'' $\mathbb{BP}$ with the ellipses $E_A$, $E_B$, $E_C$ and points $W:=(1,1,1)$, $(-1,-1,1)$, $(-1,1,-1)$, $(1,-1,-1)$ removed
is a disjoint union of its $8$ open regions. We consider some points in these regions, in order to check the number
of solutions of the Grunert system.
A general picture is shown in Fig.~\ref{pillowcase_acute}.

Recall that $N(+,-,-)=N(-,+,-)=N(-,-,+)=0$ for any non-degenerate base $ABC$ by
Lemma \ref{le.one.horn}.
Therefore, we need not check the points in $\mathbb{BP}(+,-,-)$, $\mathbb{BP}(-,+,-)$, and
$\mathbb{BP}(-,-,+)$.

Now, we examine some five points: $V_1 \in \mathbb{BP}(+,+,+)$, $V_2 \in \mathbb{BP}(-,+,+)$,
$V_3 \in \mathbb{BP}(+,+,-)$, $V_4 \in \mathbb{BP}(+,-,+)$,
$V_5 \in \mathbb{BP}(-,-,-)$.

\begin{figure}[t]
\center{\includegraphics[width=0.7\textwidth]{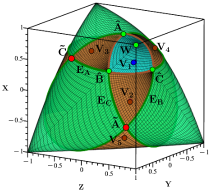}}\\%{Pillowcase_sol.pdf}{pillowcase.pdf}
\caption{The decomposition of the pillowcase for the acute-angled base.
There are two solutions in the region containing point $V_1$ and
there is one solution in each region containing
one of the points $V_2, V_3, V_4, V_5$.}
\label{pillowcase_acute}
\end{figure}

As the main partial case we consider  the case with regular base $ABC$, $d_1=d_2=d_3=1$.
Let us consider the point
\begin{equation*}
V_1=\frac{1}{10}\left(\, 7, \, \sqrt{85}, \, \sqrt{85} \right).
\end{equation*}
After some simplifications, we obtain the following Grunert equation \eqref{eq.Grunert.eq1} for this point:
\begin{equation*}
P(u)=-\frac{12}{125}\left(5u^2-\sqrt{85}u+5\right)\left(2u^2-\sqrt{85}u +5\right)=0.
\end{equation*}

It is easy to see that this equation has only two positive solutions: $u_{1}=(\sqrt{85}+3\sqrt{5})/4$  and
$u_{2}=(\sqrt{85}-3\sqrt{5})/4$.

If $u=u_{1} = (\sqrt{85}+3\sqrt{5})/4$ then
system \eqref{eq_Grunert_1} has the following solution:
$$
s_1=\frac{\sqrt{51}-3\sqrt{3}}{6}, \quad s_2=\frac{\sqrt{15}}{3}, \quad s_3=\frac{\sqrt{15}}{3}.
$$

If $u=u_{2} = (\sqrt{85}-3\sqrt{5})/4$ then
system \eqref{eq_Grunert_1} has the following solution:
$$
s_1=\frac{\sqrt{51}+3\sqrt{3}}{6}, \quad s_2=\frac{\sqrt{15}}{3}, \quad s_3=\frac{\sqrt{15}}{3}.
$$
\smallskip

Let us consider the point
\begin{equation*}
V_2=\frac{1}{2}\left(\,  0, \, {\sqrt{2}}, \, {\sqrt{2}} \right).
\end{equation*}
After some simplifications, we obtain the following Grunert equation for this point:
\begin{equation*}
P(u)=\left(u^2-\sqrt{2}u+1\right)\left(u^2+\sqrt{2}u -1\right)=0.
\end{equation*}
It is easy  to see that this equation has one positive solution $u_{1}=(\sqrt{6}-\sqrt{2})/2$
and then system \eqref{eq_Grunert_1} implies
$$
s_1=\frac{1+\sqrt{3}}{2}, \quad s_2=\frac{\sqrt{2}}{2}, \quad s_3=\frac{\sqrt{2}}{2}.
$$
\smallskip

Similarly, for each of the points
\begin{equation*}
V_3=\frac{1}{2}\left(\, {\sqrt{2}}, \, {\sqrt{2}} ,\, 0\right)
\quad
\mbox{and}
\quad
V_4=\frac{1}{2}\left(\, {\sqrt{2}}, \, 0, \, {\sqrt{2}}\right),
\end{equation*}
we get one solution:
$$
s_1=\frac{\sqrt{2}}{2}, \quad s_2=\frac{\sqrt{2}}{2}, \quad s_3=\frac{1+\sqrt{3}}{2},
$$
and
$$
s_1=\frac{\sqrt{2}}{2}, \quad s_2=\frac{1+\sqrt{3}}{2}, \quad s_3=\frac{\sqrt{2}}{2},
$$
respectively.
\smallskip

Let us consider the  point
\begin{equation*}
V_5=\frac{1}{10}\left(\, -7, \, \sqrt{15}, \, \sqrt{15} \right).
\end{equation*}
After some simplifications, we obtain the following Grunert equation for this point:
\begin{equation*}
P(u)=-\frac{2}{125}\left(5u^2-\sqrt{15}u+5\right)\left(12u^2+\sqrt{15}u -5\right)=0.
\end{equation*}
It is easy to see that this equation has only one positive solution $u_{1} = (\sqrt{255}-\sqrt{15})/24$, and
then system \eqref{eq_Grunert_1} has the following solution:
$$
s_1=\frac{\sqrt{51}+17\sqrt{3}}{34}, \quad s_2=\frac{\sqrt{85}}{17}, \quad s_3=\frac{\sqrt{85}}{17}.
$$

Therefore, we obtain the following result:

\begin{theorem}\label{te.decom_acute}
If the base $ABC$ is acute-angled, then there are two solutions to the Grunert system for any point in $\mathbb{BP}(+,+,+)$;
there is one solution for any point from  $ \mathbb{BP}(+,+,-) \bigcup \mathbb{BP}(+,-,+)\bigcup \mathbb{BP}(-,+,+) \bigcup \mathbb{BP}(-,-,-)$;
there is no solution for any point from $\mathbb{BP}(+,-,-)\bigcup \mathbb{BP}(-,+,-) \bigcup \mathbb{BP}(-,-,+)$.
\end{theorem}

\medskip

\begin{prop}\label{pr.consolgs4}
For any non-degenerate base $ABC$, we have $N(+,+,+)=2$.
\end{prop}

\begin{proof} We know this property for a regular base $A_0B_0C_0$ (see above).
Let us prove it for an arbitrary $ABC$. It clear that $h(x,y) \to (1,1,1)$ if $x^2+y^2 \to \infty$.
Let us fix $(x,y)$ such that $\breve{b}:=h(x,y)$ is situated in the sets  $\mathbb{BP}(+,+,+)$ and  $\mathbb{BP}_0(+,+,+)$ (for both bases).
By Lemma \ref{le.consolgs2}, we get $N(+,+,+)=N_0(+,+,+)=2$.
\end{proof}
\smallskip

\subsection{The base is rectangular}
Without loss of generality, we may assume that $\angle BAC \geq \angle ABC \geq \angle ACB >0$.
We know that $\mathbb{BP}(-,+,+)=\emptyset$ by Lemma \ref{le.obtusen}. It is easy to see that all other
$\mathbb{BP}(i,j,k)$ are connected region of $\mathbb{BP}$.

By Lemmas \ref{le.one.horn}, \ref{le.obtuse} and Proposition \ref{pr.consolgs4}, we get that
$N(+,+,+)=2$, $N(-,-,+)=0$, $N(-,+,-)=0$, $N(-,+,+)=0$, $N(+,-,-)=0$.

Therefore, we should find only $N(+,-,+)$, $N(+,+,-)$, and $N(-,-,-)$.
We can check only one rectangular base, say, such that $\angle BAC= \pi/2$, $\angle ABC= \angle ACB= \pi/4$, $d_1=2$, $d_2=d_3=\sqrt{2}$.

\begin{figure}[t]
\center{\includegraphics[width=0.7\textwidth]{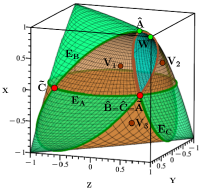}}\\
\caption{The decomposition of the pillowcase for the rectangular base.
There is one solution in each region containing
one of the points $V_1, V_2, V_3$.}
\label{pillowcase_right}
\end{figure}

Let's clarify the following information: $V_1\in\mathbb{BP(+,+,-)}$, $V_2\in\mathbb{BP(+,-,+)}$
and $V_3\in\mathbb{BP(-,-,-)}$, see Fig.~\ref{pillowcase_right}.

Let us consider the point
\begin{equation*}
V_1=\left(\,  \, 1/2, \, 1, \, 1/2 \right).
\end{equation*}
After some simplifications, we obtain the following Grunert equation \eqref{eq.Grunert.eq1} for this point:
\begin{equation*}
P(u)=-4\left(u-2\right)^2(u^2-u +1)=0.
\end{equation*}
It is easy to see that this equation has only one positive solution $u_{1}=2$  (of multiplicity~$2$).

Solving system \eqref{eq_Grunert_1} for the point $V_1$, we get the following solution:
$$
s_1=\frac{\sqrt{6}}{3}, \quad s_2=\frac{2\sqrt{6}}{3}, \quad s_3=\frac{\sqrt{6}}{3}+\sqrt{2}.
$$

Let us consider the point
\begin{equation*}
V_2=\left(\,  \, 1/2, \, 1/2, \, 1 \right).
\end{equation*}
After some simplifications, we obtain the following Grunert equation for this point:
\begin{equation*}
P(u)=-4\left(u-1\right)^2(u^2-2u -2)=0.
\end{equation*}
It is easy to see that this equation has only two positive solutions: $u_{1}=1$  (of multiplicity~$2$) and
$u_{2}=1+\sqrt{3}$.

Its first solution $u_1=1$ implies that $s_2=s_1$ and
the first equation of system \eqref{eq_Grunert_1} is not satisfied.

If $u=u_{2} = 1+\sqrt{3}$ then
system \eqref{eq_Grunert_1} has the following solution:
$$
s_1=\frac{\sqrt{6}}{3}, \quad s_2=\frac{\sqrt{6}}{3}+\sqrt{2}, \quad s_3=\frac{2\sqrt{6}}{3}.
$$

Let us consider the point
\begin{equation*}
V_3=\left(\, -1/2, \, 1/2, \, 1/2 \right).
\end{equation*}
After some simplifications, we obtain the following Grunert equation for this point:
\begin{equation*}
P(u)=-4\left(u^2-u+1\right)\left(u^2+2u -2\right)=0.
\end{equation*}
It is easy to see that this equation has only one positive solution $u_{1}=\sqrt{3}-1$.

Solving system \eqref{eq_Grunert_1} for the point $V_3$, we get the following solution:
$$
s_1=1+\frac{\sqrt{3}}{3}, \quad s_2=\frac{2\sqrt{3}}{3}, \quad s_3=\frac{2\sqrt{3}}{3}.
$$

Therefore, we obtain the following result (recall that $\mathbb{BP}(-,+,+)=\emptyset$ by Lemma~\ref{le.obtusen}):

\begin{theorem}\label{te.decom_right}
If the base $ABC$ is rectangular, then there are two solutions to the Grunert system for any point in $\mathbb{BP}(+,+,+)$;
there is one solution for any point from  $ \mathbb{BP}(+,+,-) \bigcup \mathbb{BP}(+,-,+) \bigcup \mathbb{BP}(-,-,-)$;
there is no solution for any point from $\mathbb{BP}(+,-,-)\bigcup \mathbb{BP}(-,+,-) \bigcup \mathbb{BP}(-,-,+)$.
\end{theorem}

\smallskip

\subsection{The base is obtuse-angled}

Without loss of generality, we may assume that $\angle BAC \geq \angle ABC \geq \angle ACB >0$.

We know that $\mathbb{BP}(-,+,+)=\emptyset$ by Lemma \ref{le.obtusen}.
Note that the $\mathbb{BP}(+,-,-)$ have two connected component $\mathbb{BP}(+,-,-)'$  and $\mathbb{BP}(+,-,-)''$
($\mathbb{BP}(+,-,-)'$ is closer to the point $(x_0,y_0,z_0)$ than $\mathbb{BP}(+,-,-)''$).
We denote by $N(+,-,-)'$  and $N(+,-,-)''$ the number of solution of the Grunert system for the point from $\mathbb{BP}(+,-,-)'$  and $\mathbb{BP}(+,-,-)''$ respectively.
It is easy to see that all other six sets
$\mathbb{BP}(i,j,k)$ are connected region of $\mathbb{BP}$.

By Lemmas \ref{le.one.horn}, \ref{le.obtuse} and Proposition \ref{pr.consolgs4}, we get that
$N(+,+,+)=2$, $N(-,-,+)=0$, $N(-,+,-)=0$, $N(-,+,+)=0$, $N(+,-,-)''=0$.

Therefore, we should find only $N(+,-,+)$, $N(+,+,-)$, $N(-,-,-)$, and $N(+,-,-)'$.
We can check only one obtuse-angled base, say, such that $\angle BAC= 2\pi/3$, $\angle ABC= \angle ACB= \pi/6$, $d_1=\sqrt{3}$, $d_2=d_3=1$.

\begin{figure}[t]
\center{\includegraphics[width=0.7\textwidth]{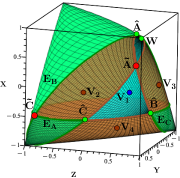}}\\
\caption{The decomposition of the pillowcase for the obtuse base.
There are two solutions in the region containing point $V_1$ and
there is one solution in each region containing
one of the points $V_2, V_3, V_4$.}
\label{pillowcase_obtuse}
\end{figure}

Let's clarify the following information: $V_1\in\mathbb{BP(+,-,-)'}$, $V_2\in\mathbb{BP(+,+,-)}$,
$V_3\in\mathbb{BP(+,-,+)}$ and $V_4\in\mathbb{BP(-,-,-)}$, see Fig.~\ref{pillowcase_obtuse}.

Let us consider the point
\begin{equation*}
V_1=\left(\,  0, \, \sqrt{2}/2, \, \sqrt{2}/2 \right).
\end{equation*}
After some simplifications, we obtain the following Grunert equation \eqref{eq.Grunert.eq1} for this point:
\begin{equation*}
P(u)=\left(u^2-3\sqrt{2}u +3\right)\left(u^2-\sqrt{2}u +1\right)=0.
\end{equation*}
It is easy to see that this equation has only two positive solutions: $u_{1}=(3\sqrt{2}+6)/2$  and
$u_{2}=(3\sqrt{2}-6)/2$.

If $u=u_{1} = (3\sqrt{2}+6)/2$ then
system \eqref{eq_Grunert_1} has the following solution:
$$
s_1=\frac{\sqrt{3}-1}{2}, \quad s_2=\frac{\sqrt{6}}{2}, \quad s_3=\frac{\sqrt{6}}{2}.
$$

If $u=u_{2} = (3\sqrt{2}-6)/2$ then
system \eqref{eq_Grunert_1} has the following solution:
$$
s_1=\frac{\sqrt{3}+1}{2}, \quad s_2=\frac{\sqrt{6}}{2}, \quad s_3=\frac{\sqrt{6}}{2}.
$$

Let us consider the point
\begin{equation*}
V_2=\left(\,  0, \, 1, \, 0 \right).
\end{equation*}
After some simplifications, we obtain the following Grunert equation for this point:
\begin{equation*}
P(u)=\left(u^2-3\right)\left(u^2+1\right)=0.
\end{equation*}
It is easy to see that this equation has only one positive solution $u_{1}=\sqrt{3}$.

Solving system \eqref{eq_Grunert_1} for the point $V_2$, we get the following solution:
$$
s_1=\frac{1}{2}, \quad s_2=\frac{\sqrt{3}}{2}, \quad s_3=\frac{3}{2}.
$$

Let us consider the point
\begin{equation*}
V_3=\left(\,  0, \, 0, \, 1 \right).
\end{equation*}
After some simplifications, we obtain the following Grunert equation for this point:
\begin{equation*}
P(u)=\left(u-1\right)^2\left(u-3\right)^2=0.
\end{equation*}
It is easy to see that this equation has only two positive solutions: $u_{1}=1$  (of multiplicity~$2$) and
$u_{2}=3$ (of multiplicity~$2$).

Its first solution $u_1=1$ implies that $s_2=s_1$ and
the first equation of system \eqref{eq_Grunert_1} is not satisfied.

If $u=u_{2} = 3$ then
system \eqref{eq_Grunert_1} has the following solution:
$$
s_1=\frac{1}{2}, \quad s_2=\frac{3}{2}, \quad s_3=\frac{\sqrt{3}}{2}.
$$

Let us consider the point
\begin{equation*}
V_4=\left(\,  -7/10, \, \sqrt{15}/10, \, \sqrt{15}/10 \right).
\end{equation*}
After some simplifications, we obtain the following Grunert equation for this point:
\begin{equation*}
P(u)=-\frac{12}{125}\left(5u^2-\sqrt{15}u+5\right)\left(2u^2+3\sqrt{15}u-15\right)=0.
\end{equation*}
It is easy to see that this equation has only one positive solution $u_{1}=(\sqrt{255}-3\sqrt{15})/4$.

Solving system \eqref{eq_Grunert_1} for the point $V_4$, we get the following solution:
$$
s_1=\frac{3\sqrt{17}+17}{34}, \quad s_2=\frac{\sqrt{255}}{17}, \quad s_3=\frac{\sqrt{255}}{17}.
$$

Therefore, we obtain the following result (recall that $\mathbb{BP}(-,+,+)=\emptyset$ by Lemma~\ref{le.obtusen}):

\begin{theorem}\label{te.decom_obtuse}
If the base $ABC$ is obtuse-angled, then there are two solutions to the Grunert system for any point in $\mathbb{BP}(+,+,+) \bigcup \mathbb{BP}(+,-,-)'$;
there is one solution for any point from  $ \mathbb{BP}(+,+,-) \bigcup \mathbb{BP}(+,-,+) \bigcup \mathbb{BP}(-,-,-) $;
there is no solution for any point from $\mathbb{BP}(+,-,-)''\bigcup \mathbb{BP}(-,+,-) \bigcup \mathbb{BP}(-,-,+)$.
\end{theorem}

\bigskip

The results obtained above provide a detailed answer to the question about the number of solutions of the Snellous--Pothenot problem for a fixed base $ABC$ and a given set of cosines of the angles $\angle BDC, \angle ADC, \angle ADB$.
Theorems \ref{te.decom_acute}, \ref{te.decom_right} and \ref{te.decom_obtuse} proved above together with Theorem~\ref{te.ellipces}
provide a complete answer to Question~\ref{prob.2} in the case of an acute-angled, rectangular and obtuse-angled base $ABC$, respectively.

We also note that the indicated results are constructive and on their basis a program can easily be written to find the number of solutions of the corresponding system of equations for the Snellius--Pothenot problem for any non-degenerate base.

\vspace{15mm}

\end{document}